 \documentclass[final,1p,times]{elsarticle}

\usepackage{amssymb}
 \usepackage{amsthm}
\usepackage{amsmath,amssymb,amsopn,amsfonts,mathrsfs,amsbsy,amscd}
\usepackage{longtable}
\usepackage{longtable}
\usepackage{caption}
\usepackage{multirow}

\newcommand{\prs}{\langle\;,\;\rangle}
\newcommand{\br}{[\;,\;]}
\newcommand{\too}{\longrightarrow}

\newcommand{\esp}{\quad\mbox{and}\quad}

\newcommand{\G}{{\mathfrak{g}}}

\newcommand{\ad}{{\mathrm{ad}}}

\newcommand{\tr}{{\mathrm{tr}}}
\newcommand{\ric}{{\mathrm{ric}}}
\newcommand{\Ri}{{\mathrm{Ric}}}

\newcommand{\B}{{\cal B}}

\newcommand{\al}{\alpha}
\newcommand{\be}{\beta}

\newcommand{\ga}{\gamma}

\newcommand{\e}{\epsilon}

\newcommand{\la}{\lambda}

\font\bb=msbm10

\def\B{\hbox{\bb B}}
\def\R{\hbox{\bb R}}

\newcommand{\n}{\mathfrak{n}}

\newtheorem{theo}{Theorem}[section]
\newtheorem{pr}{Proposition}[section]
\newtheorem{Le}{Lemma}[section]
\newtheorem{co}{Corollary}[section]

\newtheorem{exem}{Example}
\newtheorem{rem}{Remark}

\begin{document}

\begin{frontmatter}


 

\title{  On Einstein Lorentzian nilpotent Lie groups}

\author[label1,label2]{ Mohamed Boucetta, Oumaima Tibssirte}
 \address[label1]{Universit\'e Cadi-Ayyad\\
 Facult\'e des sciences et techniques\\
 BP 549 Marrakech Maroc\\e-mail: m.boucetta@uca.ma
 }
 
 \address[label2]{Universit\'e Cadi-Ayyad\\
 	Facult\'e des sciences et techniques\\
 	BP 549 Marrakech Maroc\\e-mail: oumayma1tibssirte@gmail.com 
 }



\begin{abstract}In this paper, we study Lorentzian left invariant Einstein metrics on nilpotent Lie groups. We show that if the center of such Lie groups is degenerate then they are Ricci-flat and their Lie algebras can be obtained by the double   extension process from an abelian Euclidean Lie algebra. We show that all nilpotent Lie groups up to dimension $5$ endowed with a Lorentzian Einstein left invariant metric  have degenerate center and we use this fact to give a complete classification of these metrics. We show that if $\G$ is the Lie algebra of a nilpotent Lie group endowed with a Lorentzian left invariant Einstein metric with non zero scalar curvature then the center $Z(\G)$ of $\G$ is nondegenerate Euclidean, the derived ideal $[\G,\G]$ is nondegenerate Lorentzian and $Z(\G)\subset[\G,\G]$. We give the first examples of Ricci-flat Lorentzian nilpotent Lie algebra with nondegenerate center.
\end{abstract}


\end{frontmatter}






\section{Introduction} \label{section1}
A pseudo-Riemannian manifold $(M,g)$ is called Einstein if its Ricci operator $\Ri:TM\too TM$ satisfies $\Ri=\la\mathrm{Id}_{TM}$ for some constant $\la\in\R$. When $\la=0$, $(M,g)$ is called Ricci-flat. Pseudo-Riemannian Einstein manifolds present a central topic of differential geometry and an active area of research. The subclass of Lorentzian Einstein manifolds has attracted a particular interest due to its importance in the physics of general relativity (see \cite{besse}). Homogeneous Riemannian manifolds were intensively studied and the  Alekseevskii's conjecture (see \cite{besse}) has driven a profound exploration  of Einstein left invariant Riemannian metrics on Lie groups  leading to some outstanding results (see \cite{heber, lauret1}). However, the study of left invariant Einstein pseudo-Riemannian metrics on Lie groups is at beginning. 
 In \cite{Aub-Med, bou1}, flat Lorentzian left invariant metrics on Lie groups has been studied, in \cite{bouh} flat left invariant metrics of signature $(2,n-2)$ on nilpotent Lie groups has been characterized, Ricci-flat Lorentzian left invariant metrics on 2-step nilpotent Lie groups has been investigated in \cite{bouc1, guediri} and in \cite{calvaruso3, derd}, all  four dimensional Lie algebras of Einstein Lorentzian Lie groups were given. The study of pseudo-Riemannian Einstein left invariant metric with non vanishing scalar curvature has been initiated in \cite{Dconti1}.

In this paper, we study Einstein Lorentzian left invariant metrics on nilpotent Lie groups. As in any study involving left invariant structures on Lie groups, we can consider the problem at the Lie algebra level. Let $(\G,\br,\prs)$ be a nilpotent Lorentzian Lie algebra with Ricci operator $\mathrm{Ric}:\G\longrightarrow\G$ satisfying $\mathrm{Ric}=\lambda\mathrm{Id}_\G$. Our main results can be stated as follows :
\begin{enumerate}
	\item If the center $\mathrm{Z}(\G)$ of $\G$ is nondegenerate then it is Euclidean and if the derived ideal $[\G,\G]$ is nondegenerate then it is Lorentzian. 
	\item If  $[\G,\G]$ is degenerate then $Z(\G)$ is degenerate and the metric is Ricci-flat.
	\item If the scalar curvature of $\G$ is non zero then $Z(\G)$ is nondegenerate Euclidean, $[\G,\G]$ is nondegenerate  Lorentzian and $Z(\G)\subset[\G,\G]$.
	\item If $Z(\G)$ is degenerate then $\G$ is Ricci-flat and $(\G,\br,\prs)$ is obtained by the process of double extension from an abelian Euclidean Lie algebra. The process of double extension has been introduced by Medina-Revoy \cite{medina} in the context of bi-invariant pseudo-Riemannian metrics on Lie groups and turned out to be efficient in many other situations. We adapt this process to our case and, in addition to our main result, we use it to construct a large class of Einstein Lorentzian Lie algebras (not necessarily nilpotent). We also recover the description of $2$-step nilpotent Lorentzian Lie algebras obtained in \cite{guediri}. 
	\item If $\G$ is Ricci-flat non-abelian, and $\dim[\G,\G]=\dim (Z(\G)\cap[\G,\G])+1$ then $Z(\G)$ is degenerate.
	\item If $\dim\G\leq 5$ then the center of $\G$ is degenerate. In this case we give a complete classification of all such Lie algebras.
	\item We give the first examples Ricci-flat Lorentzian nilpotent Lie algebras with nondegenerate center. It is worth to mention that this differs from the flat case. Indeed, it has been shown (see \cite{bou1, bou}) that if a nilpotent Lie group $G$ is endowed with a flat left-invariant metric which is either Lorentzian or of signature $(2,n-2)$ then its center must be degenerate.
	\item We give another proof of the main result in \cite{Dconti1} by using a formula known in the Euclidean context (see Propositions \ref{pr4}-\ref{pr5})
	
\end{enumerate}

	 The paper is organized as follows. In Section \ref{section2}, we establish two lemmas and we give a useful expression of the Ricci operator involving our main tool a family of skew-symmetric endomorphisms we call structure endomorphisms.  In Section \ref{section3}, we prove some general results on Einstein Lorentzian nilpotent Lie algebras. In Section \ref{section4}, we describe the process of double extension which permits to construct a large class of Einstein Lorentzian Lie algebras and we prove our main result (see Theorem \ref{main}), then we show that Lorentzian Einstein nilpotent Lie algebras up to dimension $5$ satisfy the hypothesis of this theorem and we give the list of such algebras. Finally, we give the first examples of Ricci-flat Lorentzian nilpotent Lie algebras with nondegenerate center. This opens widely the door for a future study of this particular class.
\section{Ricci curvature of pseudo-Euclidean Lie algebras}\label{section2}

A \emph{pseudo-Euclidean  vector space } is  a real vector space  of finite dimension $n$
endowed with  a
nondegenerate symmetric inner product  of signature $(q,n-q)=(-\ldots-,+\ldots+)$.  When
the
signature is $(0,n)$
(resp. $(1,n-1)$) the space is called \emph{Euclidean} (resp. \emph{Lorentzian}).

Let $(V,\prs)$ be a pseudo-Euclidean vector space of signature  $(q,n-q)$. A vector $u\in V$  is called \emph{isotropic} if  $\langle u,u\rangle=0$.
 A family
$(u_1,\ldots,u_s)$ of vectors in $V$ is called \emph{orthogonal}  if, for $i,j=1,\ldots,s$
and $i\not=j$, $\langle u_i,u_j\rangle=0$. An orthonormal basis of $V$ is an orthogonal basis $(e_1,\ldots,e_n)$ such that $\e_i=\langle e_i,e_i\rangle=\pm1$. A \emph{pseudo-Euclidean basis} of $V$ is a basis $(e_1,\bar e_2,\ldots,e_q,\bar
e_q,f_1,\ldots,f_{n-2q})$ for which the non vanishing products are 
$$\langle \bar e_i, e_i\rangle=\langle f_j,f_j\rangle=1,\;i\in\{1,\ldots,q\}\;\;\;\text{and}\;\;\;j\in\{1,\ldots,n-2q \}.$$

A vector subspace $F$ of $V$ is called
 nondegenerate if $F\cap F^\perp=\{0\}$,
	 degenerate if $F\cap F^\perp\not=\{0\}$
	and totally isotropic if $\langle u,v\rangle=0$ for any $u,v\in F$.
Note that we have always $\dim F+\dim F^\perp=\dim V$  and
 if $F$ is totally isotropic then $\dim F\leq\min(q,n-q)$.
In particular, for any vector subspace $F$, $\dim (F \cap F^\perp)\leq \min(q,n-q)$. The vector subspace $F^\perp$ is the orthogonal of $F$ with respect to $\prs$.

Let $(V,\prs)$ be a Lorentzian vector space and $F\subset V$ is a vector subspace. Then either:

\begin{enumerate}
	\item $F$ is nondegenerate Euclidean and $F^\perp$ is nondegenerate Lorentzian,
	\item $F$ is nondegenerate Lorentzian and $F^\perp$ is nondegenerate Euclidean,
	\item $F$ is degenerate and $\dim (F\cap F^\perp)=1$.
\end{enumerate}

 For any endomorphism $A:V\too V$, we denote by $A^*:V\too V$ its adjoint with respect to $\prs$. The two following lemmas will be very useful later.
\begin{Le}\label{le1} Let $(V,\prs)$ be a Lorentzian vector space,  $e$ an isotropic vector and $A$ a skew-symmetric endomorphism. Then $\langle Ae,Ae\rangle\geq0$. Moreover,  $\langle Ae,Ae\rangle=0$ if and only if $Ae=\al e$ with $\al\in\R$.\end{Le}

\begin{proof} We choose an isotropic vector $\bar{e}$ such that $\langle e,\bar{e}\rangle=1$ and an orthonormal basis $(f_1,\ldots,f_r)$ of $\{e,\bar{e}\}^\perp$. Since $A$ is skew-symmetric,  we have 
	\[ Ae=\al e+\sum_{i=1}^ra_if_i\esp \langle Ae,Ae\rangle=\sum_{i=1}^ra_i^2, \]
	and the result follows.	
\end{proof}

\begin{Le}\label{le} Let $(V,\prs)$ be a Lorentzian vector space, $e$ an isotropic vector and $A$ a skew-symmetric endomorphism such that $A(e)=0$. Then: 
	\begin{enumerate}
		\item $\tr(A^2)\leq0$,
		\item $\tr(A^2)=0$ if and only if for any $x\in e^\perp$, $A(x)=\la(x)e$ and in this case $\tr(A\circ B)=0$ for any skew-symmetric endomorphism satisfying $B(e)=0$.
	\end{enumerate}
	
\end{Le}

\begin{proof} We choose a Lorentzian basis $\B=(e,\bar{e},f_1,\ldots,f_n)$ of $V$ such that $(e,f_1,\ldots,f_n)$ is a basis of $\{e\}^\perp$, $\bar{e}$ is isotropic, $\langle e,\bar{e}\rangle=1$ and $(f_1,\ldots,f_n)$ is an orthonormal basis of $\{e,\bar{e}\}^\perp$. Remark first that the restriction of $\prs$ to $\{e\}^\perp$ is nonnegative and, for any $x\in\{e\}^\perp$, $\langle x,x\rangle=0$ if and only if $x=\al e$. Now
	\[ \tr(A^2)=\langle A^2(e),\bar{e}\rangle+\langle A^2(\bar{e}),{e}\rangle
	-\sum_{i=1}^n\langle Af_i,Af_i\rangle=-\sum_{i=1}^n\langle Af_i,Af_i\rangle. \]	
	By using our remark, we deduce that $\tr(A^2)\leq0$ and $\tr(A^2)=0$ if and only if $Af_i=\al_if_i$ for $i=1,\ldots,n$. In this case, if $B$ is skew-symmetric and $B(e)=0$ then
	\[ \tr(A\circ B)=-\langle B(e),A(\bar{e})\rangle-\langle A(\bar{e}),B({e})\rangle
	-\sum_{i=1}^n\langle Af_i,Bf_i\rangle=0. \]	
\end{proof}

The study of Einstein left invariant pseudo-Riemannian metrics on Lie groups reduces to the study of their Lie algebras endowed with the corresponding pseudo-Euclidean product. We will refer to a Lie
algebra endowed with a nondegenerate symmetric inner
product as a \emph{pseudo-Euclidean Lie algebra}.

Let $(\G,\br,\prs)$ be a pseudo-Euclidean Lie algebra. Its  \emph{ Levi-Civita product}  is the bilinear map $\G\times\G\longrightarrow\G$, $(u,v)\mapsto u\cdot v$ given by the Koszul formula
\begin{eqnarray}\label{levicivita}2\langle
u\cdot v,w\rangle&=&\langle[u,v],w\rangle+\langle[w,u],v\rangle+
\langle[w,v],u\rangle.\end{eqnarray} We denote by 
$\mathrm{L}_u,\mathrm{R}_u:\G\too \G$ the corresponding left and right multiplication given by $\mathrm{R}_v(u)=\mathrm{L}_uv=uv$.
For any $u,v\in\G$, $\mathrm{L}_{u}:\G\too\G$ is skew-symmetric and $\ad_u=\mathrm{L}_{u}-\mathrm{R}_{u}$. The curvature of $\G$ is given by
\begin{eqnarray*}
	\label{curvature}K(u,v)w&=&\mathrm{L}_{[u,v]}w-[\mathrm{L}_{u},\mathrm{L}_{v}]w\\
	&=&[\mathrm{R}_w,\mathrm{L}_u](v)-\mathrm{R}_w\circ \mathrm{R}_u(v)+\mathrm{R}_{uw}(v).
\end{eqnarray*} From the last relation, we deduce that the 
 Ricci curvature $\mathrm{ric}:\G\times\G\too\R$ of $(\G,\br,\prs)$ is given by
\begin{eqnarray}\label{ric}
	\ric(u,v)&=&-\tr(\mathrm{R}_u\circ \mathrm{R}_v)+\tr(\mathrm{R}_{uv}).
\end{eqnarray} 
Since $\prs$ is non-degenerate, we can define the Ricci operator $\mathrm{Ric}:\G\longrightarrow\G$ by the expression 
$$\langle \mathrm{Ric}(u),v\rangle:=\ric(u,v).$$
To get a more useful formula of the Ricci curvature, we introduce $H\in\G$ and $J:\G\too \mathrm{so}(\G,\prs)$ such that for any $u,v\in\G$,
\begin{equation}\label{jh}
\langle H,u\rangle=\tr(\ad_u)\esp J_u(v)=\ad_v^*(u).
\end{equation}
Note that $H\in[\G,\G]^\perp$ and $H=0$ if and only if $\G$ is unimodular.

\begin{pr} We have
	\begin{eqnarray*}
		\ric(u,v)&=&-\frac12\tr(\ad_u\circ\ad_v)-\frac12\tr(\ad_u\circ
		\ad_v^*)-\frac14\tr(J_u\circ J_v)\\&&-\frac12{\langle}\ad_Hu,v{\rangle}-
		\frac12{\langle}\ad_Hv,u{\rangle}.\nonumber	
	\end{eqnarray*}
	
\end{pr}

\begin{proof} It is a consequence of \eqref{ric}, the following formula which can be deduced from \eqref{levicivita} $$\mathrm{R}_u=-\frac12\left(\ad_u+\ad_u^*\right)-\frac12J_u,$$and the following computation. For any orthonormal basis $(e_1,\ldots,e_n)$ of $\G$,
	\begin{eqnarray*}
	\tr(\mathrm{R}_{uv})&=&\sum_{i=1}^n\e_i\langle \mathrm{L}_{e_i}(uv),e_i\rangle\\
	&\stackrel{\eqref{levicivita}}=&-\sum_{i=1}^n\e_i\langle [uv,e_i],e_i\rangle\\
	&=&-\tr(\ad_{uv})\\
	&=&-\langle H,uv\rangle\\
	&=&-\frac12{\langle}\ad_Hu,v{\rangle}-
	\frac12{\langle}\ad_Hv,u{\rangle}.
	\end{eqnarray*}
	\end{proof}
	
	\begin{pr}\label{pr} If $\G$ is nilpotent  then
		\begin{equation*}
			\ric(u,v)=-\frac12\tr(\ad_u\circ
			\ad_v^*)-\frac14\tr(J_u\circ J_v).	
		\end{equation*}	In particular, its Ricci operator $\Ri:\G\too\G$ is given by		
		\begin{equation} \label{ricci1} \mathrm{Ric}=-\frac12{\mathcal
			J}_1+\frac14{\mathcal J}_2, \end{equation} where $\mathcal{J}_1$ and $\mathcal{J}_2$ are the auto-adjoint endomorphisms given by 
		\begin{equation}\label{12} \langle \mathcal{J}_1(u),v\rangle=\tr(\ad_u\circ
		\ad_v^*)\esp  \langle \mathcal{J}_2(u),v\rangle=-\tr(J_u\circ
		J_v). \end{equation} 
	\end{pr}

\begin{rem}\label{rem1} The endomorphisms $J_u$ are skew-symmetric and $J_u=0$ if and only if $u\in[\G,\G]^\perp$. So if $\prs$ is Euclidean then, for any $u\in\G$, $\langle \mathcal{J}_i(u),u\rangle\geq0$ $(i=1,2)$, $\ker \mathcal{J}_1=Z(\G)$ and $\ker \mathcal{J}_2=[\G,\G]^\perp$.
\end{rem}

The operators ${\mathcal J}_1$ and ${\mathcal J}_2$ will play a crucial role in our study so we are going to express them in a useful way. This is based on the notion of structure endomorphisms we now introduce.

Let $(\G,\br,\prs)$ be a pseudo-Euclidean Lie algebra and
$(e_1,\ldots,e_p)$ a basis of
$\G$. For any $u,v\in\G$, the Lie bracket can be written
\begin{equation}\label{bracket}[u,v]=\sum_{i=1}^p\langle S_iu,v\rangle
e_i,\end{equation}where $S_i:\G\too\G$ are skew-symmetric
endomorphisms with respect to $\prs$. 
 The family $(S_1,\ldots,S_p)$ will be called
\emph{structure endomorphisms} associated to $(e_1,\ldots,e_p)$. Note that $Z(\G)=\cap_{i=1}^p\ker S_i$
and one can see  easily from \eqref{bracket} and the definition of $J$ in \eqref{jh}  that, for
any
$u\in\G$,
\begin{equation}\label{J}J_u=\sum_{i=1}^p\langle u,e_i\rangle S_i\end{equation}
The  following proposition will be very useful later.
\begin{pr}Let $(\G,\prs)$ be a pseudo-Euclidean   Lie algebra,
	$(e_1,\ldots,e_p)$  a basis of
	$\G$ and $(S_1,\ldots,S_p)$ the corresponding structure endomorphisms. Then 
	\begin{equation}\label{invariant}
	{\mathcal J}_1=-\sum_{i,j=1}^p\langle e_i,e_j\rangle S_i\circ S_j\quad
	\mbox{and}\quad {\mathcal J}_2u=-\sum_{i,j=1}^p\langle e_i,u\rangle{\tr}(S_i\circ
	S_j)e_j.\end{equation}In particular, $\tr{\mathcal J}_1=\tr{\mathcal J}_2$.

\end{pr}
\begin{proof} The expression of ${\mathcal J}_2$ is an immediate consequence of \eqref{J}.
 We have
\begin{eqnarray*}
	\ad_u\circ\ad_u^*v&=&\ad_u\circ J_vu\\
	&\stackrel{(\ref{bracket})}=&\sum_{i=1}^p\langle S_iu,J_vu\rangle e_i\\
	&\stackrel{(\ref{J})}=&\sum_{i,j}\langle S_iu,S_ju\rangle \langle v,e_j\rangle e_i\\
	&=&-\sum_{i,j}\langle (S_j\circ S_i)(u),u\rangle K_{i,j}v,\end{eqnarray*}where
$K_{i,j}v=\langle v,e_j\rangle e_i$. Clearly $\tr(K_{i,j})=\langle e_i,e_j\rangle$, thus $\tr(\ad_u\circ\ad_v^*)=-\sum_{i,j}\langle S_j\circ S_i u,v\rangle\langle
e_i,e_j\rangle$ which gives the desired formula of ${\mathcal J}_1$.
\end{proof}

\section{Some results on  Einstein Lorentzian nilpotent Lie algebras}\label{section3}

In this section,  we will show  that for an Einstein Lorentzian nilpotent Lie algebra if $[\G,\G]$ is nondegenerate (resp. $Z(\G)$ is nondegenerate) then it is Lorentzian (resp. it is Euclidean). If $[\G,\G]$ is degenerate then $[\G,\G]\cap[\G,\G]^\perp\subset Z(\G)$ and $\G$ is Ricci flat. If $\dim[\G,\G]=\dim([\G,\G]\cap Z(\G))+1$ and $\G$ carries a Ricci flat Lorentzian metric then $Z(\G)$ is degenerate. Finally, we will use our approach to recover some results proved in \cite{Dconti1}.

Before going further, let us give the following remark which we will use frequently. Let $(\G,\br,\prs)$ be a  pseudo-Euclidean Lie algebra. From \eqref{jh}, one can easily deduce that $\ker J=[\G,\G]^\perp$ and hence
\begin{equation}\label{re} Z(\G)\subset\ker\mathcal{J}_1:=M\esp [\G,\G]^\perp\subset\ker\mathcal{J}_2:=N. \end{equation}
Since  $\mathcal{J}_1$ and $\mathcal{J}_2$ are symmetric
\begin{equation}\label{re1} \mathrm{Im}\mathcal{J}_1=M^\perp\subset Z(\G)^\perp\esp \mathrm{Im}\mathcal{J}_2=N^\perp\subset [\G,\G].  \end{equation}

\begin{pr}\label{pr1} Let $(\G,\br,\prs)$ be an Einstein Lorentzian nilpotent non abelian Lie algebra. If $[\G,\G]$ is non degenerate then it is Lorentzian.
	
\end{pr}

\begin{proof} Suppose that $[\G,\G]$ is nondegenerate Euclidean,  choose an orthonormal basis $(e_1,\ldots,e_d)$ of $[\G,\G]$ and denote by $(S_1,\ldots,S_d)$ the associated structure endomorphisms. According to \eqref{ricci1} and  \eqref{invariant}, we have 
	\[ -\frac12\mathcal{J}_1+\frac14\mathcal{J}_2=\la\mathrm{Id}_\G,\;
	\mathcal{J}_1=-\sum_{i=1}^dS_i^2 \esp \mathcal{J}_2u=-\sum_{i,j=1}^d\langle u,e_i\rangle\tr(S_i\circ S_j)e_j.    \]
	Since $\G$ is nilpotent then $\dim[\G,\G]^\perp\geq2$ and we can choose a couple $(e,\bar{e})$  of isotropic vectors in $[\G,\G]^\perp$ such that $\langle e,\bar{e}\rangle=1$. By replacing in the relations above and using \eqref{re}, we get
	\[\frac12\mathcal{J}_1e=-\la e,\; \frac12\mathcal{J}_1\bar{e}=-\la \bar{e}\esp \sum_{i=1}^d\langle S_ie,S_ie\rangle=\sum_{i=1}^d\langle S_i\bar{e},S_i\bar{e}\rangle=0. \]
	By using Lemma \ref{le1}, we deduce that for any $i\in\{1,\ldots,d \}$, $S_ie=\al_i e$ and $S_i\bar{e}=-\al_i \bar{e}$ and hence
	\[ \la=\frac12\sum_{i=1}^d\al_i^2\geq0. \]
	For $i=1,\ldots,d$, $S_i$ is skew-symmetric and leaves invariant $\mathrm{span}\{e,\bar{e}\}$ so it leaves invariant its orthogonal.
	We denote by $K_i$ the restriction of $S_i$ to the Euclidean vector space $\{e,\bar{e}\}^\perp$. We have $\tr(S_i^2)=2\al_i^2+\tr(K_i^2)$ and $\tr(K_i^2)\leq0$. Now, since $\tr(\mathcal{J}_1)=\tr(\mathcal{J}_2)$, we get
	\[ (\dim\G)\la=-\frac14\tr(\mathcal{J}_1)=\frac14\sum_{i=1}^d(2\al_i^2+\tr(K_i^2))=\la+\frac14\sum_{i=1}^d\tr(K_i^2). \]This shows that $\la\leq0$. By summing up the results we obtained, we deduce that $\la=0$, $\tr(K_i^2)=0$ and $\alpha_i^2=0$ which implies that $S_i=0$ for $i=1,\ldots,d$.
	Thus $\G$ is abelian which is a contradiction and completes the proof.
\end{proof}

\begin{pr}\label{pr2} Let $(\G,\prs)$ be an Einstein pseudo-Euclidean non abelian nilpotent Lie algebra. If $Z(\G)$ is nondegenerate then $Z(\G)^\perp$  is not Euclidean.
	
\end{pr}

\begin{proof} Denote by $(p,q)=(-,\ldots-,+,\ldots,+)$ the signature of $\prs$.
	Suppose that $Z(\G)$ is nondegenerate and $Z(\G)^\perp$ is Euclidean. Then we can choose  an orthogonal family  
	$(e_1,\ldots,e_p)$ in $Z(\G)$ such that $\langle e_i,e_i\rangle=-1$ for $i=1,\ldots,p$. We have $\G=\mathrm{span}\{e_1,\ldots,e_p\}\oplus\G_0$, where $\G_0=\{e_1,\ldots,e_p\}^\perp$.
	For any $u,v\in\G_0$, put 
	\begin{equation}\label{g0}[u,v]=\sum_{i=1}^p\langle K_iu,v\rangle e_i+[u,v]_0,\end{equation} where $K_i:\G_0\too\G_0$ are skew-symmetric and $[u,v]_0\in\G_0$. Let $\prs_0$ be the restriction of $\prs$ to $\G_0$.
It is obvious that $(\G_0,\br_0,\prs_0)$ is an Euclidean nilpotent Lie algebra.
			We claim that if $(\G,\prs)$ is Einstein then $\la=\frac14\tr(K_i^2)\leq0$ for $i=1,\ldots,p$. Moreover
		$$\Ri_{\prs_0}=\la\mathrm{Id}_{\G_0}+\frac12\sum_{i=1}^pK_i^2.\eqno(*)$$

	This implies that the Ricci curvature of $(\G_0,\prs_0)$ is nonpositive. But a non abelian nilpotent Euclidean Lie algebra has always a Ricci negative direction and a Ricci positive direction (see \cite[Theorem 2.4]{milnor}).  So the only possibility is  $K_i=0$ for $i=1,\ldots,p$ and $\G_0$ is abelian. We get a contradiction which completes the proof.

	 Let us prove our claim. We choose an orthonormal basis $\B_1=(f_1,\ldots,f_q)$ of $\G_0$. Then $\B=(e_1,\ldots,e_p,f_1,\ldots,f_q)$ is an orthonormal basis of $\G$. Denote by $(S_1,\ldots,S_p,T_1,\ldots,T_q)$ the structure endomorphisms of $(\G,\prs)$ with respect to $\B$ and $(M_1,\ldots,M_q)$ the structure endomorphisms of $(\G_0,\prs_0)$ with respect to $\B_1$. The $S_i$ and the $T_i$ vanish on $Z(\G)$ and hence leave invariant $\G_0$. By using \eqref{g0}, one can easily see that, for $i=1,\ldots,p$ and $j=1,\ldots,q$,
	 \[ (S_i)_{|\G_0}=K_i\esp (T_j)_{|\G_0}=M_j. \]  
	  If $\G$ is Einstein then, according to  \eqref{invariant}, we have
	\[ -\frac12\sum_{i=1}^p S_i^2+\frac12\sum_{i=1}^q T_i^2+\frac14\mathcal{J}_2=\la\mathrm{Id}_\G,\eqno(**) \] where
	\[ \mathcal{J}_2=-\sum_{i,j}\langle e_i,\bullet\rangle\tr(K_i\circ K_j)e_j
	-\sum_{i,j}\langle f_i,\bullet\rangle\tr(M_i\circ M_j)f_j 
	-\sum_{i,j}\langle f_i,\bullet\rangle\tr(M_i\circ K_j)e_j-\sum_{i,j}\langle e_i,\bullet\rangle\tr(K_i\circ M_j)f_j.\]
	If we evaluate the relation $(**)$ at $e_i$, we get
	\[ \frac14\tr( K_i^2)e_i+\frac14\sum_{j=1}^q\tr( K_i\circ M_j )f_j=\la e_i. \]	
	This is equivalent to $\la=\frac14\tr( K_i^2)$ and $\tr( K_i\circ M_j )=0$ for $i\in\{1,\ldots,p \}$ and $i\in\{1,\ldots,q \}$. This implies that if we restrict $(**)$ to $\G_0$ we get the desired relation $(*)$.	
\end{proof}

\begin{co}\label{co}	Let $(\G,\br,\prs)$ be an Einstein Lorentzian non abelian nilpotent Lie algebra. If $Z(\G)$ is nondegenerate then it is Euclidean.\end{co}
The following result was first found by Guediri in \cite{guediri} and served as a key ingredient in the classification of Einstein Lorentzian $2$-step nilpotent Lie algebras. It can also be deduced from the preceding results. 
\begin{co}\label{co0} Let  $(\G,\br,\prs)$ be an Einstein Lorentzian non abelian 2-step nilpotent Lie algebra. Then $Z(\G)$ is  degenerate. 
	
\end{co}

\begin{proof} Suppose that $Z(\G)$ is nondegenerate. According to Corollary \ref{co}, $Z(\G)$ is nondegenerate Euclidean. But $\G$ is 2-step nilpotent and hence $[\G,\G]\subset Z(\G)$. Thus $[\G,\G]$ is nondegenerate Euclidean which contradicts Proposition \ref{pr1}.
	\end{proof}

\begin{pr}\label{pr3} Let $(\G,\prs)$ be an Einstein Lorentzian nilpotent Lie algebra such that $[\G,\G]$ is degenerate then $[\G,\G]\cap[\G,\G]^\perp\subset Z(\G)$ and $(\G,\prs)$ is Ricci flat.
	
\end{pr}

\begin{proof} Let $e$ be a generator of $[\G,\G]\cap[\G,\G]^\perp$. Then there exists a basis
	 $(e,\bar{e},f_1,\ldots,f_{d},g_1,\ldots,g_s)$  of $\G$ such that $(e,f_1,\ldots,f_{d})$ is a basis of $[\G,\G]$, $(e,g_1,\ldots,g_s)$ is  basis of $[\G,\G]^\perp$, $\bar{e}$ is isotropic, $\langle e,\bar{e}\rangle=1$ and $(f_1,\ldots,f_{d},g_1,\ldots,g_s)$ is an orthonormal basis of $\{ e,\bar{e}\}^\perp$.
	 Denote by $(A,S_1,\ldots,S_d)$ the associated structure endomorphisms, i.e., for any $u,v\in\G$,
	\[ [u,v]=\langle Au,v\rangle e+\sum_{i=1}^d\langle S_iu,v\rangle f_i. \]
	According to \eqref{ricci1} and  \eqref{invariant}, we have
	\[ -\frac12\mathcal{J}_1+\frac14\mathcal{J}_2=\la\mathrm{Id}_{\G}\esp \mathcal{J}_1=-\sum_{i=1}^dS_i^2. \] Since $e\in[\G,\G]^\perp$ and it is isotropic, we have $\mathcal{J}_2e=0$, 
	$-\frac12\mathcal{J}_1e=\la e,$ and hence
	$\sum_{j=1}^d\langle S_je,S_je\rangle=0.$
	By using Lemma \ref{le1}, we get $S_je=a_j e$ for any $j=1,\ldots,d$ and hence
	$\la=\frac12\sum_{i=1}^da_i^2.$ On the other hand, since $\tr(\mathcal{J}_1)=\tr(\mathcal{J}_2)$, we get
	\[ (\dim\G) \la=-\frac14\tr(\mathcal{J}_1)=\frac14\sum_{j=1}^d\tr(S_j^2). \]
	On the other hand
	\begin{eqnarray*}
		\tr(S_j^2)&=&\langle S_j^2e,\bar{e}\rangle+\langle S_j^2\bar{e},{e}\rangle+
		\sum_{l}\langle S_j^2f_l,f_l\rangle+\sum_{l}\langle S_j^2g_l,g_l\rangle\\
		&=&2a_j^2-\sum_{l}\langle S_jf_l,S_jf_l\rangle-\sum_{l}\langle S_jg_l,S_jg_l\rangle.
	\end{eqnarray*}Since $S_j$ leaves invariant $e$, it leaves invariant its orthogonal $\mathrm{span}\{e,f_l,g_k\}$. But the restriction of $\prs$ to $\mathrm{span}\{e,f_l,g_k\}$ is nonnegative. So
	 $\langle S_jf_l,S_jf_l\rangle\geq0$ and $\langle S_jg_l,S_jg_l\rangle\geq0$. Thus
	\[ (\dim\G-1)\la=-\sum_{l,j}\langle S_jf_l,S_jf_l\rangle-\sum_{l,j}
	\langle S_jg_l,S_jg_l\rangle\leq0. \]But we have shown so far that $\la\geq0$. In conclusion,  $\la=0$ and 
	$S_j(e)=0$ for $j=1,\ldots,p$.  This implies that, for any $u\in\G$,
	$[e,u]=\langle A(e),u\rangle e.$ But $\ad_u$ is nilpotent an hence $[e,u]=0$ which completes the proof.
\end{proof}

\begin{co}
Let $\G$ be a nilpotent Lorentzian Einstein Lie algebra. Suppose that $[\G,\G]$ is degenerate, then $\mathrm{Z}(\G)$ is also degenerate.
\end{co}

\begin{pr}\label{pr6} Let $(\G,\prs)$ be a Ricci flat Lorentzian nilpotent non abelian Lie algebra such that $\dim[\G,\G]=\dim\left( Z(\G)\cap[\G,\G]  \right)+1$. Then $Z(\G)$ is degenerate.
	
\end{pr}

\begin{proof} Suppose that $Z(\G)$ is nondegenerate. Then, according to Proposition \ref{pr1}, Proposition \ref{pr3} and Corollary \ref{co}, $Z(\G)$ is Euclidean and $[\G,\G]$ is Lorentzian and hence there exists an orthonormal basis $(e_1,\ldots,e_r)$ of $[\G,\G]$ such that $e_i\in Z(\G)$ for $i=1,\ldots,r-1$ and $\langle e_r,e_r\rangle=-1$. We denote by $(S_1,\ldots,S_r)$ the structure endomorphisms associated to  $(e_1,\ldots,e_r)$. We have
	\[ -\frac12\mathcal{J}_1+\frac14\mathcal{J}_2=0,\;\mathcal{J}_1=S_r^2-\sum_{j=1}^{r-1}S_i^2\esp \mathcal{J}_2(u)=
	-\sum_{i,j}\langle e_i,u\rangle\tr(S_i\circ S_j)e_j. \]
	Since $Z(\G)\subset \ker\mathcal{J}_1$, we get $\mathcal{J}_2(e_i)=0$ for $i=1,\ldots,r-1$. This is equivalent to $\tr(S_i\circ S_j)=0$ for $i=1,\ldots,r$ and $j=1\ldots,r-1$ and hence
	\[ \mathcal{J}_2(u)=\langle e_r,u\rangle\tr(S_r^2)e_r. \] 
	But $\tr(\mathcal{J}_1)=\tr(\mathcal{J}_2)=0$ so $\mathcal{J}_1=\mathcal{J}_2=0$. This implies, by virtue of \eqref{12}, that 
	$\tr(\ad_x\circ\ad_y^*)=0$ for any $x,y\in\G$. For $x\in\G$, put
	\[ \ad_x(e_r)=\al_1e_1+\ldots+\al_r e_r. \]So  $\ad_x^2(e_r)=\al_r^2e_r$ and since $\ad$ is nilpotent then $\al_r=0$ and hence for any $x\in\G$, $\ad_x(e_r)\in Z(\G)$. If $(f_1,\ldots,f_q)$ is an orthonormal basis of $[\G,\G]^\perp$, then
	\begin{eqnarray*}
		0&=&\tr(\ad_{e_r}\circ\ad_{e_r}^*)\\
		&=&\sum_{i=1}^{r-1}\langle \ad_{e_r}(e_i),\ad_{e_r}(e_i)\rangle+\sum_{i=1}^q\langle \ad_{e_r}(f_i),\ad_{e_r}(f_i)\rangle\\
		&=&\sum_{i=1}^q\langle \ad_{f_i}(e_r),\ad_{f_i}(e_r)\rangle.
	\end{eqnarray*}But $\ad_{f_i}(e_r)\in Z(\G)$ and $Z(\G)$ is Euclidean thus $\ad_{f_i}(e_r)=0$ for $i=1,\ldots,q$ an hence $e_r\in Z(\G)$ which is a contradiction. This completes the proof.
\end{proof}

By using our approach, we recover some results obtained in \cite{Dconti1}.

\begin{pr} Let $(\G,\br,\prs)$ be a nilpotent pseudo-Euclidean Lie algebra. Then
	\begin{enumerate}	\item If $(\G,\br,\prs)$ is Einstein with $\la\not=0$ then $Z(\G)\subset[\G,\G]$.
		\item  If $\dim Z(\G)\geq\dim [\G,\G]$ then $(\G,\br,\prs)$ is Einstein if and only if it is Ricci flat.
	\end{enumerate}In particular, if $\G$ is 2-nilpotent then $(\G,\br,\prs)$ is Einstein if and only if it is Ricci flat.
\end{pr}

\begin{proof} Suppose that $(\G,\br,\prs)$ is nilpotent and Einstein with $\la\not=0$, i.e.,
	\[ -\frac12\mathcal{J}_1+\frac14\mathcal{J}_2=\la\mathrm{Id}_\G. \]
	This implies, by virtue of \eqref{re} and \eqref{re1},
	\[ Z(\G)\subset \mathrm{Im}\mathcal{J}_2\subset [\G,\G]. \]
	It implies also $M\cap N=\{0\}$. But, if $\dim Z(\G)\geq\dim[\G,\G]$ then
	\[ \dim M+\dim N\geq \dim Z(\G)+\dim[\G,\G]^\perp\geq\dim\G\]and hence
	$\G=M\oplus N$. This contradict $\tr(\mathcal{J}_1)=tr(\mathcal{J}_2)$.
	\end{proof}

One of the main results in \cite{Dconti1} is that if a pseudo-Euclidean Einstein nilpotent Lie algebra has a derivation with a non vanishing trace then it is Ricci flat. We give another proof of this fact based on \eqref{formula}. This formula  was established in the Euclidean  context in \cite{lauret} by using the Ricci tensor as a moment map. We prove this formula in the general case by a direct computation.

\begin{pr}\label{pr4} Let $(\G,\prs)$ be a pseudo-Euclidean Lie algebra and let $Q$ denote the symmetric endomorphism $Q=-\frac12\mathcal{J}_1+\frac14\mathcal{J}_2$. Then for any orthonormal basis $(e_1,\ldots,e_p)$ of $\G$ and any endomorphism $E$ of $\G$, we have
	\begin{equation}\label{formula} \tr(QE)=\frac14\sum_{i,j}\e_i\e_j\langle E([e_i,e_j])-[E(e_i),e_j]-[e_i,E(e_j)],[e_i,e_j]\rangle, \end{equation}where $\langle e_i,e_i\rangle=\e_i$.
	
\end{pr}
\begin{proof} We denote by $(S_1,\ldots,S_p)$ the structures endomorphisms associated to $(e_1,\ldots,e_p)$. From \eqref{J}, we get $S_i=\e_i J_{e_i}$ and by using
	 \eqref{invariant} we get
	\[ QE(u)=\frac12\sum_{i=1}^p\e_i J_{e_i}^2E(u)-\frac14\sum_{i,j=1}^p\e_i\e_j\langle e_i,E(u)\rangle{\tr}(J_{e_i}\circ
	J_{e_j})e_j. \]Let us compute:
	\begin{eqnarray*}
		\tr(QE)&=&\sum_{j=1}^p\e_j\langle QE(e_j),e_j\rangle\\
		&=&-\frac12\sum_{i,j=1}^p\e_i\e_j\langle J_{e_i}E(e_j),J_{e_i}(e_j)\rangle -\frac14\sum_{i,j=1}^p\e_i\e_j\langle e_i,E(e_j)\rangle{\tr}(J_{e_i}\circ
		J_{e_j})\\
		&=&-\frac12\sum_{i,j=1}^p\e_j\e_i\langle e_i,[E(e_j),J_{e_i}(e_j)]\rangle+\frac14\sum_{i,j,l=1}^p\e_i\e_l\e_j\langle e_i,E(e_j)\rangle\langle J_{e_j}e_l,J_{e_i}e_l\rangle\\
		&=&-\frac12\sum_{i,j,l=1}^p\e_j\e_i\e_l\langle J_{e_i}(e_j),e_l\rangle \langle e_i,[E(e_j),e_l]\rangle+\frac14\sum_{j,l=1}^p\e_l\e_j\langle J_{e_j}e_l,J_{E(e_j)}e_l\rangle\\
		&=&-\frac12\sum_{i,j,l=1}^p\e_j\e_i\e_l\langle {e_i},[e_j,e_l]\rangle \langle e_i,[E(e_j),e_l]\rangle+\frac14\sum_{i,j,l=1}^p\e_l\e_j\e_i\langle J_{e_j}e_l,e_i\rangle \langle e_i,J_{E(e_j)}e_l\rangle\\
		&=&-\frac12\sum_{j,l=1}^p\e_j\e_l \langle [e_j,e_l],[E(e_j),e_l]\rangle+\frac14\sum_{i,j,l=1}^p\e_l\e_j\e_i\langle {e_j},[e_l,e_i]\rangle \langle [e_l,e_i],{E(e_j)}\rangle\\
		&=&-\frac12\sum_{j,l=1}^p\e_j\e_l \langle [e_j,e_l],[E(e_j),e_l]\rangle+\frac14\sum_{i,l=1}^p\e_l\e_i \langle [e_l,e_i],{E([e_l,e_i])}\rangle\\
		&=&	-\frac14\sum_{j,l=1}^p\e_j\e_l \langle [e_j,e_l],[E(e_j),e_l]\rangle-\frac14\sum_{j,l=1}^p\e_j\e_l \langle [e_j,e_l],[e_j,E(e_l)]\rangle+\frac14\sum_{i,l=1}^p\e_l\e_i \langle [e_l,e_i],{E([e_l,e_i])}\rangle,
	\end{eqnarray*}and the formula follows.
\end{proof}

From Proposition \ref{pr4} we get:

\begin{pr}(\cite[Theorem 4.1]{Dconti1})\label{pr5} Let $(\G,\prs)$ be a pseudo-Euclidean  nilpotent Lie algebra having a derivation with non zero trace. Then $(\G,\prs)$ is Einstein if and only if it is Ricci flat.
\end{pr}
\begin{proof}
Let $D\in\mathrm{Der}(\G)$ such that $\mathrm{tr}(D)\neq 0$. Write $\mathrm{Ric}=\lambda\mathrm{Id}_\G$, using formula \eqref{ricci1} and \eqref{formula} we get that $\lambda\mathrm{tr}(D)=0$ and therefore $\lambda=0$.
\end{proof}
\begin{rem} The derivations of nilpotent Lie algebras have been widely studied and computed (see \cite{goze}). It turns out that nilpotent Lie algebras having a derivation with non null trace are the most common. For instance, any nilpotent Lie algebra up to dimension 6 has this property and most of the nilpotent Lie algebras of dimension 7 have this property (see \cite{Dconti1}). 
	
\end{rem}

\section{ Einstein Lorentzian nilpotent Lie algebras with degenerate center}\label{section4}
In this section, we give a complete description of Einstein Lorentzian nilpotent Lie algebras with degenerate center. We will show that these Lie algebras are obtained by a double extension process of an abelian Euclidean Lie algebra. The double extension process was introduced by Medina-Revoy in \cite{medina} in the context of quadratic Lie algebras. It turned out to be useful in many other situations. We give here a version of this process adapted to our study.

Consider $(V,\prs_0)$ an Euclidean vector space, $b\in V$, $K,D:V\too V$ two endomorphisms of $V$ such that $K$ is skew-symmetric. We endow the vector space $\G=\R e\oplus V\oplus\R \bar{e}$  with the inner product
$\prs$ which extends $\prs_0$, so that $\mathrm{span}\{e,\bar{e}\}$ and $V$ are
orthogonal,
$e$ and $\bar{e}$ are isotropic and satisfy $\langle e,\bar{e}\rangle=1$.
We also define on
$\G$ the bracket 
\begin{equation}
\label{model} 
[\bar{e},e]=\mu e,\;\;\; [\bar{e},u]=D(u)+\langle
b,u\rangle_0 e\esp[u,v]=\langle K(u),v\rangle_0e,\quad u,v\in V.\end{equation}
 
 \begin{pr} 
 	Suppose that $(\G,\prs,\br)$ is obtained by a double extension process from a Euclidean vector space $(V,\prs_0)$ with parameters $(K,D,\mu,b)$, then 
 	\label{doubleext} 
	\begin{enumerate}
 		\item[$(i)$] $(\G,\br)$ is a Lie algebra if and only if : 
 		$$KD+D^*K=\mu K.$$ In this case $(\G,\br)$ is nilpotent if and only if $\mu=0$ and $D$ is nilpotent.
 		\item[$(ii)$]  $(\G,[\;,\;],\prs)$ is an Einstein Lorentzian Lie algebra if and only if
 		\[KD+D^*K=\mu K\esp 4\mu\tr(D)=\tr(K^2)+2\tr(D^2)+2\tr(DD^*). \]In this case, it is Ricci flat.
 	\end{enumerate}  
 \end{pr}
 \begin{proof} The bracket $[\;,\;]$ is a Lie bracket if and only if, for any $v,w\in V$, 
 	$$ [\bar{e},[v,w]]+[w,[\bar{e},v]]+[v,[w,\bar{e}]]=\langle (\mu K-K\circ D-D^*\circ K)(v),w\rangle_0 e=0.$$
 	Therefore, $(\G,[\;,\;])$ is a Lie algebra if and only if $\mu K=K\circ D+D^*\circ K$ and it is easy to see that $(\G,[\;,\;])$
is nilpotent if and only if $\mu=0$ and $D$ is a nilpotent endomorphism.
 
 Now we will compute the Ricci curvature of $(\G,\br,\prs)$ by using the formula
 \[ \ric(u,v)=-\frac12 B(u,v)-\frac12\langle \mathcal{J}_1(u),v\rangle+
 \frac14\langle \mathcal{J}_2(u),v\rangle-\frac12\langle\ad_Hu,v\rangle 
 --\frac12\langle\ad_Hv,u\rangle,\]where $B$ is the Killing form and $H$ is the vector defined in \eqref{jh}.
 
 We choose an orthonormal basis $(f_1,\ldots,f_n)$ of $V$ and we denote by $(K_0,\bar{K},S_1,\ldots,S_n)$ the  structure endomorphisms of $(e,\bar{e},f_1,\ldots,f_n)$. By a direct computation, we get that $B$ and $H$ are given by
 \[ H=(\mu+\tr(D))e,\; \R e\oplus V\subset\ker B\esp B(\bar{e},\bar{e})=\mu^2+\tr(D^2). \]On the other hand, $\bar{K}=0$ and for any $u,v\in\G$
 \[ \langle K_0(u),v\rangle=\langle[u,v],\bar{e}\rangle\esp 
 \langle S_i(u),v\rangle=\langle[u,v],f_i\rangle,\;u,v\in\G, i=1,\ldots,n. \]
 This gives that
 \[ \begin{cases}K_0(e)=-\mu e,K_0(\bar{e})=\mu \bar{e}+b,\; K_{0}(f_i)=K(f_i),\\
 S_i(e)=0, S_i(f_j)=-\langle D^*(f_i),f_j\rangle e \esp S_i(\bar{e})=D^*(f_i).
 \end{cases}
  \]
 From these relations, one can easily deduce that $\tr(K_0\circ S_i)=\tr(S_i\circ S_j)=0$ for $i,j=1,\ldots,n$ and hence
 \[\mathcal{J}_1=-\sum_{i=1}^n S_i^2\esp {\mathcal
 	J}_2=-\langle e,\bullet\rangle\tr(K_0^2)e. \]
 Using these expressions, a careful computation gives
 \[ \R e\oplus V\subset\ker\ric\esp \ric(\bar{e},\bar{e})
 =-\frac12\tr(D^2)-\frac12\tr(DD^*)-\frac14\tr(K^2)+\mu\tr(D). \]
 This completes the proof.	
 	 \end{proof}
 Any data $(K,D,\mu,b)$ satisfying the conditions in Proposition \ref{doubleext} is called \emph{admissible}. We can now state the main theorem of this section, which gives the structure of Einstein Lorentzian nilpotent Lie algebras with degenerate center.

\begin{theo}\label{main} Let $(\G,\prs)$ be an Einstein nilpotent non abelian Lorentzian Lie algebra and suppose that there exists $e\in Z(\G)$  a central isotropic vector. Then:
	\begin{enumerate}\item $Z(\G)$ is degenerate and $\G$ is Ricci-flat.
		\item $\G$ is obtained from $\G_0$ by the double extension process with  admissible data $(K,D,0,b)$ and $D$ is nilpotent.
	\end{enumerate} 
	
\end{theo}

\begin{proof} Denote $\mathcal{I}=\R e$ and choose an orthonormal basis $\B=(e,\bar{e},f_1,\ldots,f_n)$ of $\G$ such that $(e,f_1,\ldots,f_n)$ is a basis of $\mathcal{I}^\perp$, $\bar{e}$ is isotropic, $\langle e,\bar{e}\rangle=1$ and $(f_1,\ldots,f_n)$ is an orthonormal basis of $\{e,\bar{e}\}^\perp$. We denote by $(K,\bar{K},S_1,\ldots,S_n)$ the structure endomorphisms of $\B$, i.e., for any $u,v\in\G$,
	\[ [u,v]=\langle Ku,v\rangle e+\langle \bar{K}u,v\rangle \bar{e}+\sum_{i=1}^n\langle S_iu,v\rangle f_i. \]
	 Moreover, according to \eqref{ricci1} and  \eqref{invariant}, we have 
\begin{equation}\label{p}\begin{cases} -\frac12\mathcal{J}_1+\frac14\mathcal{J}_2=\la\mathrm{Id}_\G,\; {\mathcal
	J}_1=-\bar{K}\circ K-K\circ\bar{K}-\sum_{j=1}^nS_j^2,\\
 {\mathcal
 	J}_2=-\left[\langle e,\bullet\rangle\tr(K^2)+\sum_{i=1}^n\tr(K\circ S_i)\langle f_i,\bullet\rangle\right]e-\left[\langle \bar{e},\bullet\rangle\tr(\bar{K}^2)+\sum_{i=1}^n\tr(\bar{K}\circ S_i)\langle f_i,\bullet\rangle\right]\bar{e}
 -\langle e,\bullet\rangle\tr(K\circ\bar{K})\bar{e}\\-\langle \bar{e},\bullet\rangle\tr(K\circ\bar{K}){e}-\sum_{i=1}^n\langle e,\bullet\rangle\tr(K\circ S_i)f_i-\sum_{i=1}^n\langle \bar{e},\bullet\rangle\tr(\bar{K}\circ S_i)f_i-\sum_{i,j}\langle f_i,\bullet\rangle\tr(S_i\circ S_j)f_j.
\end{cases}
\end{equation}	 
Since $e\in Z(\G)$, then $K(e)=\bar{K}(e)=S_i(e)=0$ for $i=1,\ldots,n$ and $\mathcal{J}_1(e)=0$. This implies that $\frac14\mathcal{J}_2(e)=\la e$ which is equivalent to
\[ \frac14\tr(K\circ\bar{K})=-\la\esp \tr(\bar{K}^2)=\tr(\bar{K}\circ S_i)=0\;\mbox{for}\; i=1,\ldots,n. \]
According to Lemma \ref{le}, for any $x\in\mathcal{I}^\perp$, $\bar{K}(x)=\al(x)e$ and $-4\la=\tr(\bar{K}\circ K)=0$.
 On other hand, $\tr(\mathcal{J}_1)=\tr(\mathcal{J}_2)$ and from the first relation in \eqref{p} we deduce that   $\tr(\mathcal{J}_1)=-\sum_{i=1}^n\tr(S_i^2)=0$. By using Lemma \ref{le}, we also deduce that $\tr(S_i^2)=0$, for any $x\in\mathcal{I}^\perp$, $S_i(x)=s_i(x)e$ and $\tr(K\circ S_i)=\tr(S_i\circ S_j)=0$ for $i,j\in\{1,\ldots,n\}$. By skew-symmetry, we deduce that, for $j=1,\ldots,n$
 \[ \bar{K}(\bar{e})=-\sum_{i=1}^n\al(f_i)f_i\esp S_j(\bar{e})=-\sum_{i=1}^ns_j(f_i)f_i. \]
  On the other hand, for any $u\in \mathcal{I}^\perp$,
 \[ [\bar{e},u]=\langle K(\bar{e}),u\rangle e-\al(u)\bar{e}-\sum_{i=1}^ns_i(u)f_i. \]But $\ad_u$ is nilpotent and then we must have $\al(u)=0$ for any $u\in\mathcal{I}^\perp$ and hence $\bar{K}=0$.
 To sum up, if we put $V=\mathrm{span}\{f_1,\ldots,f_n\}$ and we define $D:V\too V$ by $D(u)=\sum_{i=1}^n\langle S_i(\bar{e}),u\rangle f_i$, then 
 \[ \begin{cases}\;[u,v]=\langle Ku,v\rangle e, u,v\in V,\\
 \;[\bar{e},u]=\langle K(\bar{e}),u\rangle e+D(u),\;u\in V,\\
 \mathcal{J}_2=\langle e,\bullet\rangle\tr(K^2)e,\\
 -\frac12\mathcal{J}_1+\frac14\mathcal{J}_2=0,\; {\mathcal
 	J}_1=-\sum_{j=1}^nS_j^2.
\end{cases} \]This completes the proof.
	\end{proof}
As an application of Theorem \ref{main} we recover the following results due to Guediri \cite[Lemma $14$ and Theorem $15$]{guediri} :
\begin{co}\label{co1} Let $(\G,\br,\prs)$ be an Einstein Lorentzian 2-step nilpotent Lie algebra. Then $Z(\G)$ is  degenerate and $\G$ is Ricci-flat.
\end{co}
\begin{theo}
	Let $\n$ be a $2$-step nilpotent, non-abelian Lie algebra. Then $\G$ admits a Ricci-flat Lorentzian metric if and only if $\G=\R^n\oplus\n$ (a direct sum of Lie algebras) such that $\n$ is a Lie algebra for which the Lie brackets are expressed in a basis $\mathcal{B}=\{e,z_1,\dots,z_p,\bar{e},e_1,\dots,e_q\}$ as follows :
	\begin{equation}
	\label{guedeq0}
	[\bar{e},e_i]=\alpha_i e+\sum_{k=1}^p c_{ik}z_k,\;\;\;\;[e_i,e_j]=a_{ij}e,\;\;\;\;1\leq i,j\leq q,
	\end{equation}
	with $\sum\limits_{i,j=1}^q a_{ij}^2=2\sum\limits_{i=1}^q\sum\limits_{k=1}^p c_{ik}^2$. Moreover the basis $\mathcal{B}$ can be chosen Lorentzian, in particular the restriction of the metric to $[\G,\G]$ is degenerate.
\end{theo}

\begin{proof} 
	Suppose that $(\G,[\;,\;],\langle\;,\;\rangle)$ is a $2$-step nilpotent Lorentzian, Ricci-flat Lie algebra. By virtue of Corollary \ref{co1},  $\mathrm{Z}(\G)$ is degenerate and   Theorem \ref{main} implies that $\G$ is given by a process of double extension from a Euclidean vector space $V_0$ with parameters $(K,D,0,b)$, i.e $\G=\R e\oplus V_0\oplus \R \bar{e}$ where $e,\bar{e}$ are isotropic vectors satisfying $\langle \bar{e},e\rangle=1$ and, for any $u,v\in V_0$, 
	\begin{equation}
	\label{guedeq1}
	[\bar{e},u]=D(u)+\langle b,u\rangle e,\;\;\;[u,v]=\langle K(u),v\rangle e.
	\end{equation}
	Moreover, Proposition \ref{doubleext} implies that $D^2=0$ and
	\begin{equation}
	\label{guedeq2}
	K\circ D+D^*\circ K=0,\;\;\;\;\;2\mathrm{tr}(DD^*)=-\mathrm{tr}(K^2).
	\end{equation}
	First, we observe that $\mathrm{Im}(D)\subset \mathrm{Z}(\G)\cap V_0$. Indeed, given $w\in\G$ and $u\in V_0$ we have that : $$[w,Du]=[w,Du+\langle b,u\rangle e]=[w,[\bar{e},u]]=0.$$
	Write $V_0=(V_0\cap\mathrm{Z}(\G))\overset{\perp}{\oplus}W_0$ and $V_0\cap\mathrm{Z}(\G)=\mathrm{Im}(D)\overset{\perp}{\oplus} S$, then $S$ is an abelian Lie subalgebra of $\G$ since it is contained in $\mathrm{Z}(\G)$ and we have that $\G=\R^n\oplus\n$ with $\n=\R e\oplus \R\bar{e}\oplus \mathrm{Im}(D)\oplus W_0$, moreover using \eqref{guedeq1} we can check that $\n$ is a Lie subalgebra of $\G$. Next, let $\{z_1,\dots,z_p\}$ be a Euclidean basis of $\mathrm{Im}(D)$ and let $\{e_1,\dots,e_q\}$ be a Euclidean basis of $W_0$. Write :
	$$D(e_i)=\sum_{k=1}^p c_{ik}z_k,\;\;\;\langle b,e_i\rangle=\alpha_i,\;\;\;\;\langle K(e_i),e_j\rangle=a_{ij}.$$
	Then it follows that :
	$$[\bar{e},e_i]=\alpha_i e+\sum_{k=1}^p c_{ik}z_k,\;\;\;\;[e_i,e_j]=a_{ij}e,\;\;\;\;1\leq i,j\leq q.$$
	Now 
	$$\mathrm{tr}(K^2)=-\sum\limits_{i,j=1}^q a_{ij}^2=2,\; \mathrm{tr}(DD^*)=\sum\limits_{i=1}^q\sum\limits_{k=1}^p c_{ik}^2$$ 
	and we conclude by \eqref{guedeq2}. Conversely, for any $2$-step nilpotent Lie algebra of the form $\G=\R^n\oplus\n$ satisfying \eqref{guedeq0}, we have $\mathrm{Ric}_{\G}=0$ follows from a straightforward calculation.
\end{proof}
\section{Classification of Einstein Lorentzian nilpotent Lie algebra of dimension $\leq 5$} 
In this  section, we give a complete description of the Lorentzian Lie algebras associated to all Einstein Lorentzian nilpotent Lie group of dimension $\leq5$. This is based on Theorem \ref{main} and the following result.
\begin{theo}\label{five} Let $(\G,\br,\prs)$ be an Einstein Lorentzian nilpotent Lie group of dimension $\leq5$. Then the center of $\G$ is degenerate.

\end{theo}

\begin{proof} We use the classification of nilpotent Lie algebras up to dimension 6 given by \cite{graaf}.In Table \ref{tab:my-table}, we give the list of nilpotent Lie algebras up to dimension 5 and for each  of them we give a derivation with a non zero trace. We will also use Corollary \ref{co1} and Proposition \ref{pr6}.
	
	There is a unique nilpotent Lie algebra in dimension $3$ which is $L_{3,2}$ and it is 2-step nilpotent hence we can apply Corollary \ref{co}. In dimension 4, there is two  nilpotent Lie algebras namely $L_{4,2}$ whose center is degenerate by Corollary \ref{co1} and $L_{4,3}$ whose Lie bracket is given by
	\[ [e_{1},e_{2}]= e_{3},\;[e_{1},e_{3}]= e_{4}. \]
	It is clear that $\mathrm{L}_{4,3}$ satisfies the hypothesis of Proposition \ref{pr6}. \\
	
	We can see that apart from $L_{5,6}$ and $L_{5,7}$ all the other Lie algebras are either 2-step nilpotent or satisfy the hypothesis of Proposition \ref{pr6}. Let us now study $L_{5,6}$ and $L_{5,7}$. 
	
	If we denote by $\G$ either $L_{5,6}$ or $L_{5,7}$, one can see that \begin{equation}\label{5} Z(\G)\subset[\G,[\G,\G]]\subset[\G,\G],\; \dim Z(\G)=1,\; \dim[\G,[\G,\G]]=2  \esp \dim[\G,\G]=3.\end{equation} To complete the proof of the theorem, we will show that if a five dimensional nilpotent Lie algebra $\G$ satisfies \eqref{5} and have an Einstein Lorentzian metric then its center must be degenerate.

Let $(\G,\br,\prs)$ be a five dimensional Einstein Lorentzian nilpotent Lie algebra satisfying \eqref{5} such that its center nondegenerate. Note first that according to Prosition \ref{pr5}, $(\G,\br,\prs)$ must be Ricci flat.

According to Corollary \ref{co} and Propositions \ref{pr1} and \ref{pr3}, $Z(\G)$ must be Euclidean and $[\G,\G]$ must be nondegenerate Lorentzian. We distinguish three cases. \\\\	
$1.$ {\bf $[\G,[\G,\G]]$ is non degenerate Euclidean}. Then we can choose an orthonormal  basis $(f_1,f_2,f_3,f_4,f_5)$ such that $\langle f_3,f_3\rangle=-1$, $Z(\G)=\R f_5$, $[\G,[\G,\G]]=\mathrm{span}\{f_4,f_5\}$ and $[\G,\G]=\mathrm{span}\{f_3,f_4,f_5\}$. So
	\[ \begin{cases}
	[f_1,f_2]=af_3+bf_4+cf_5,[f_1,f_3]=df_4+xf_5,[f_1,f_4]=yf_5,\\ [f_2,f_3]=zf_4+tf_5,[f_2,f_4]=uf_5,\;[f_3,f_4]=vf_5,\;\quad a\not=0,(z,d)\not=(0,0).
	\end{cases} \]
	This bracket satisfies the Jacobi identity if and only if $v=0$ and $yz-du=0$. The Ricci operator is given by\\\\
	{\footnotesize
		{\[\displaystyle \frac12 \left[ \begin {array}{ccccc} {a}^{2}-{b}^{2}-{c}^{2}+{d}^{2}+{x}^{2}-{y}^{2}&dz+xt-yu&zb+ct&cu&0\\ \noalign{\medskip}dz+xt-yu&{a}^{2}-{b}^{2}-{c}^{2}+{z}^{2}+{t}^{2}-{u}^{2}&-bd-cx&-cy&0\\ \noalign{\medskip}-zb-ct&bd+cx&-{a}^{2}+{d}^{2}+{x}^{2}+{z}^{2}+{t}^{2}&ab+xy+tu&ac\\ \noalign{\medskip}cu&-cy&-ab-xy-tu&{b}^{2}-{d}^{2}-{y}^{2}-{z}^{2}-{u}^{2}&bc-dx-zt\\ \noalign{\medskip}0&0&-ac&bc-dx-zt&{c}^{2}-{x}^{2}+{y}^{2}-{t}^{2}+{u}^{2}\end {array} \right]. \]}}
	Since $a\not=0$ then $c=0$ and hence the Ricci operator is given by
	{\small
		{\[\displaystyle \frac12 \left[ \begin {array}{ccccc} {a}^{2}-{b}^{2}+{d}^{2}+{x}^{2}-{y}^{2}&dz+xt-yu&zb&0&0\\ \noalign{\medskip}dz+xt-yu&{a}^{2}-{b}^{2}+{z}^{2}+{t}^{2}-{u}^{2}&-bd&0&0\\ \noalign{\medskip}-zb&bd&-{a}^{2}+{d}^{2}+{x}^{2}+{z}^{2}+{t}^{2}&ab+xy+tu&\\ \noalign{\medskip}0&0&-ab-xy-tu&{b}^{2}-{d}^{2}-{y}^{2}-{z}^{2}-{u}^{2}&-dx-zt\\ \noalign{\medskip}0&0&&-dx-zt&-{x}^{2}+{y}^{2}-{t}^{2}+{u}^{2}\end {array} \right]. \]}}
	
	The couple $(z,d)\not=(0,0)$ otherwise $\dim[\G,\G]\leq2$, hence $b=0$. So
	{\small
		{\[\displaystyle \frac12 \left[ \begin {array}{ccccc} {a}^{2}+{d}^{2}+{x}^{2}-{y}^{2}&dz+xt-yu&0&0&0\\ \noalign{\medskip}dz+xt-yu&{a}^{2}+{z}^{2}+{t}^{2}-{u}^{2}&0&0&0\\ \noalign{\medskip}0&0&-{a}^{2}+{d}^{2}+{x}^{2}+{z}^{2}+{t}^{2}&xy+tu&\\ \noalign{\medskip}0&0&-xy-tu&-{d}^{2}-{y}^{2}-{z}^{2}-{u}^{2}&-dx-zt\\ \noalign{\medskip}0&0&&-dx-zt&-{x}^{2}+{y}^{2}-{t}^{2}+{u}^{2}\end {array} \right]. \]}}
	So we must have  $\Ri_{4,4}=-{d}^{2}-{y}^{2}-{z}^{2}-{u}^{2}=0$ and $\Ri_{2,2}={a}^{2}+{z}^{2}+{t}^{2}-{u}^{2}=0$, but then $a=0$ which is impossible.\\
	$2.$ {\bf $[\G,[\G,\G]]$ is nondegenerate Lorentzian}. As in the previous case,  we can choose an orthonormal  basis $(f_1,f_2,f_3,f_4,f_5)$ such that $\langle f_4,f_4\rangle=-1$ and $Z(\G)=\R f_5$, $[\G,[\G,\G]]=\mathrm{span}\{f_4,f_5\}$ and $[\G,\G]=\mathrm{span}\{f_3,f_4,f_5\}$. So
	\[ \begin{cases}
	[f_1,f_2]=af_3+bf_4+cf_5,[f_1,f_3]=df_4+xf_5,[f_1,f_4]=yf_5,\\ [f_2,f_3]=zf_4+tf_5,[f_2,f_4]=uf_5,\;[f_3,f_4]=vf_5,\;\quad a\not=0,(z,d)\not=(0,0).
	\end{cases} \]The Jacobi identity is given by $ bv-ud+yz=av=0$, hence $v=0$. Thus the Ricci operator is given by\\\\{\footnotesize
		{\[\displaystyle \frac12 \left[ \begin {array}{ccccc} -{a}^{2}+{b}^{2}-{c}^{2}+{d}^{2}-{x}^{2}+{y}^{2}&dz-xt+yu&-zb+ct&cu&0\\ \noalign{\medskip}dz-xt+yu&-{a}^{2}+{b}^{2}-{c}^{2}+{z}^{2}-{t}^{2}+{u}^{2}&bd-cx&-cy&0\\ \noalign{\medskip}-zb+ct&bd-cx&{a}^{2}+{d}^{2}-{x}^{2}+{z}^{2}-{t}^{2}&-ab-xy-tu&ac\\ \noalign{\medskip}-cu&cy&ab+xy+tu&-{b}^{2}-{d}^{2}+{y}^{2}-{z}^{2}+{u}^{2}&bc+dx+zt\\ \noalign{\medskip}0&0&ac&-bc-dx-zt&{c}^{2}+{x}^{2}-{y}^{2}+{t}^{2}-{u}^{2}\end {array} \right] \]}}
	So we get $b=c=0$ and hence
	The Ricci operator is given by{\small
		{\[\displaystyle \frac12 \left[ \begin {array}{ccccc} -{a}^{2}+{d}^{2}-{x}^{2}+{y}^{2}&dz-xt+yu&0&0&0\\ \noalign{\medskip}dz-xt+yu&-{a}^{2}+{z}^{2}-{t}^{2}+{u}^{2}&0&0&0\\ \noalign{\medskip}0&0&{a}^{2}+{d}^{2}-{x}^{2}+{z}^{2}-{t}^{2}&-xy-tu&0\\ \noalign{\medskip}0&0&xy+tu&-{d}^{2}+{y}^{2}-{z}^{2}+{u}^{2}&dx+zt\\ \noalign{\medskip}0&0&0&-dx-zt&{x}^{2}-{y}^{2}+{t}^{2}-{u}^{2}\end {array} \right] \]}}
	Now $0=\Ri_{3,3}+\Ri_{4,4}+\Ri_{5,5}=\frac12a^2$ and hence $a=0$ which is impossible.\\
	$3.$ {\bf $[\G,[\G,\G]]$ is  degenerate }. Then we can choose a basis $(f_1,f_2,f_3,f_4,f_5)$ the metric in this basis is given by
	$$\mathrm{Diag}\left[1,1,\left(\begin{matrix}
	0&1\\1&0
	\end{matrix} \right),1\right],$$
	 and $Z(\G)=\R f_5$, $[\G,[\G,\G]]=\mathrm{span}\{f_4,f_5\}$ and $[\G,\G]=\mathrm{span}\{f_3,f_4,f_5\}$. So
	\[ \begin{cases}
	[f_1,f_2]=af_3+bf_4+cf_5,[f_1,f_3]=df_4+xf_5,[f_1,f_4]=yf_5,\\ [f_2,f_3]=zf_4+tf_5,[f_2,f_4]=uf_5,\;[f_3,f_4]=vf_5,\;\quad a\not=0,(z,d)\not=(0,0).
	\end{cases} \]The Jacobi identity is given by $ bv-ud+yz=av=0$. Hence $v=0$. The Ricci operator is given by
	{\[\displaystyle  \frac12\left[ \begin {array}{ccccc} -2\,ab-{c}^{2}-2\,xy&-yt-xu&az+ct&cu&0\\ \noalign{\medskip}-yt-xu&-2\,ab-{c}^{2}-2\,tu&-cx-ad&-cy&0\\ \noalign{\medskip}cu&-cy&ab-xy-tu&{a}^{2}-{y}^{2}-{u}^{2}&ac\\ \noalign{\medskip}az+ct&-cx-ad&{b}^{2}-{x}^{2}-{t}^{2}&ab-xy-tu&bc+dy+zu\\ \noalign{\medskip}0&0&bc+dy+zu&ac&{c}^{2}+2\,xy+2\,tu\end {array} \right]. \]}
	So $c=d=z=0$ which is impossible.
\end{proof}
As a consequence of Theorem \ref{main} and Theorem \ref{five}, we can give the complete classification of Ricci flat Lorentzian metrics on nilpotent Lie algebras of dimension $\leq 5$. We will also make use of the following Lemma :
\begin{Le}\label{KD} Let $(V,\prs)$ be a Euclidean vector space, $K$ and $D$ two endomorphisms of $V$ such that $K$ is skew-symmetric.  Then  $KD+D^*K=0$ if and only if there exists a vector subspace $F$ of $V$ and linear maps $D_1:F\too F$, $D_2:F^\perp\too F$, $K_0,S:F^\perp\too F^\perp$ where $K_0$ is skew-symmetric invertible, $S$ symmetric and for any $u\in V$,
	\[ Du=\left\{\begin{matrix}
	D_1(u)&\mbox{if}&u\in F,\\ D_2(u)+K_0^{-1}S(u)&\mbox{if}&u\in F^\perp
	\end{matrix}   \right. \esp 
	Ku=\left\{\begin{matrix}
	0&\mbox{if}&u\in F,\\ K_0(u)&\mbox{if}&u\in F^\perp.
	\end{matrix}   \right. \]
	
\end{Le}

\begin{proof} Suppose that $KD+D^*K=0$ and put $F=\ker K$. Obviously $D(F)\subset F$,  $K(F^\perp)\subset F^\perp$ and the restriction $K_0$ of $K$ to $F^\perp$ is skew-symmetric invertible. Denote by $D_1$ the restriction of $D$ to $F$ and put for any $u\in F^\perp$, $Du=D_2u+D_3u$ where $D_2u\in F$ and $D_3u\in F^\perp$. Then
	\[ 0=K(D_2u+D_3u)+D^*K_0(u)=K_0D_3u+D_3^*K_0(u). \]
	Thus $K_0D_3=S$ where $S:F^\perp\too F^\perp$ is a symmetric endomorphism and $D_3=K_0^{-1}S$. The converse is obviously true. 	
\end{proof}

\begin{theo} \label{34} Let $(\G,\br,\prs)$ be a Ricci-flat nilpotent Lie algebra of dimension $\leq4$. Then
\begin{enumerate}
	\item[$(i)$] If $\dim\G=3$ then $\G$ is isomorphic to $(L_{3,2},\prs_{3,2})$ with the  $\prs_{3,2}=\al e^*_1\odot e^*_3+e^*_2\otimes e^*_2$ with $\al>0$. This metric is actually flat.
	\item[$(ii)$] If $\dim\G=4$ then $\G$ is isomorphic to $(L_{4,2},\prs_{4,2})$ with  $$\prs_{4,2}=\al e^*_1\odot e^*_3+e^*_2\otimes e^*_2+e^*_4\otimes e^*_4+a e^*_2\odot e^*_4,\quad \al\not=0,|a|<1,$$ or  to $(L_{4,3},\prs_{4,3})$ with  $$\prs_{4,3}=e^*_1\otimes e^*_1+a e^*_1\odot e^*_2+(a^2+b^2) e^*_2\otimes e^*_2+be^*_2\odot e^*_3+\e e^*_2\odot e^*_4+e^*_3\otimes e^*_3,\; a,b\in\R,\e=\pm1.$$
	The metric $\prs_{4,2}$ is flat and $\prs_{4,3}$ is flat if and only if $\e=-1$.
\end{enumerate}	
\end{theo}

\begin{proof} Let $(\G,\br,\prs)$ be a Einstein Lorentzian nilpotent non abelian Lie algebra of dimension $\leq5$. According Theorems \ref{main} and \ref{five}, $\G=\R e\oplus V\oplus \R\bar{e}$, where $(V,\prs_0)$ is a Euclidean vector space.  The Lie brackets are given by
	\[ [\bar{e},u]=Du+\langle b,u\rangle_0 e\esp [u,v]=\langle Ku,v\rangle_0 e, \quad u,v\in V,\] $e$ is central,   $b\in V,$ $K,D:V\too V$ with $K$ skew-symmetric, $D$ is nilpotent, $KD+D^*K=0$ and $\tr(K^2)=-2\tr(D^*D)$ furthermore metric $\prs$ satisfies $\prs_{|V}=\prs_0$, $e$ and $\bar{e}$ are isotropic in duality and orthogonal to $V$.
	
	\begin{enumerate}
		\item $\dim \G=3$ and $\dim V=1$. Then $K=D=0$ and hence $(\G,\br,\prs)$ is isomorphic $(L_{3,2},\prs_{3,2})$  where $\prs_{3,2}=\al e^*_1\odot e^*_3+e^*_2\otimes e^*_2$ and $\al>0$. This metric is flat.
		\item $\dim \G=4$ and $\dim V=2$. We distinguish two cases.
		\begin{itemize}

			\item If $K=0$ then $D=0$ and there exists a Lorentzian basis $(\bar{e},e,f_1,f_2)$ of $\G$ such that
			\[ [\bar{e},f_1]=\al e\esp [\bar{e},f_2]=\be e,\quad\al\not=0. \]  
			Put $$(e_1,e_2,e_3,e_4)=(\e\bar{e},f_1,|\al| e,\mu^{-1}(f_2-\frac{\be}{\al}f_1)),$$ where $\e$ is the sign of $\al$ and $\mu=||f_2-\frac{\be}{\al}f_1||$.
			Thus $(\G,\br,\prs)$ is isomorphic to $(L_{4,2},\prs_{4,2})$ with the metric
			$$\prs_{4,2}=\al e^*_1\odot e^*_3+e^*_2\otimes e^*_2+e^*_4\otimes e^*_4+a e^*_2\odot e^*_4,\quad \al\not=0$$ and $a=\frac{\be_1}{\sqrt{1+\be_1^2}}$ where $\be_1=\frac{\be}{\al}$. So $|a|<1$.
			
			\item	If $K\not=0$ then, according to Lemma \ref{KD},  $D=K^{-1}S$ where $S$ is symmetric. Since $D$ must be nilpotent then the rank of $S$ is equal to 1 and  there exists an orthonormal basis $\B_0=(f_1,f_2)$ of $V$ such that the matrices of $K$, $S$ and $D$ are given by
			\[ M(S,\B_0)=\mathrm{Diag}(0,s),\; M(K,\B_0)=\left(\begin{matrix}
			0&-\al\\\al&0
			\end{matrix}  \right)\esp M(D,\B_0)=\left(\begin{matrix}
			0&s\al^{-1}\\0&0
			\end{matrix}  \right),\;\al>0. \] Put $c=s\al^{-1}$. The condition $\tr(K^2)=-2\tr(D^*D)$ gives $c=\e\al$ with $\e=\pm1$.
			Thus the Lie brackets are given by
			\[ [\bar{e},f_1]=\ga e,\;[\bar{e},f_2]=\e\al f_1+\mu e\esp [f_1,f_2]=\al e. \]Put $$(e_1,e_2,e_3,e_4)=(f_2,-\e\al^{-1}\bar{e}+af_1+bf_2,f_1,-\al e)$$ with $a=\e\mu\al^{-2}$ and $b=-\e\ga\al^{-2}$.
			Then $(\G,\br,\prs)$ is isomorphic to $(L_{4,3},\prs_{4,3})$.\qedhere
			
		\end{itemize}
	\end{enumerate}
	
\end{proof}

\begin{theo}\label{Five} Let $(\G,\br,\prs)$ be a Ricci-flat nilpotent Lie algebra of dimension $5$. Then
	 $(\G,\br,\prs)$ is isomorphic to one of the following Lie algebras:
		\begin{enumerate}
			\item[$(a)$] $(L_{5,2},\prs_{5,2})$ with 
			$$\prs_{5,2}=	\al e^*_1\odot e^*_3+e^*_2\otimes e^*_2+e^*_4\otimes e^*_4+e^*_5\otimes e^*_5+a e^*_2\odot e^*_4+b e^*_2\odot e^*_5+ab e^*_4\odot e^*_5,\;\al\not=0,|a|<1,|b|<1.$$ This metric is flat.
			\item[$(b)$] $(L_{5,8},\prs_{5,8})$ with 
			\begin{eqnarray*}
				\prs_{5,8}&=&e^*_1\otimes e^*_1+ae^*_1\odot e^*_2-yx^{-1}e^*_1\odot e^*_3+(b-ayx^{-1})e^*_2\odot e^*_3+(a^2+b^2)e^*_2\otimes e^*_2\\	
				&&+\sqrt{x^2+y^2} e^*_2\odot e^*_5+(1+(yx^{-1})^2)e^*_3\otimes e^*_3+x^2e^*_4\otimes e^*_4,\;(x\not=0,a,b,y\in \R).
			\end{eqnarray*}

			\item[$(c)$] 	$(L_{5,9},\prs_{5,9})$ with  
			\begin{eqnarray*}
				\prs_{5,9}
				&=&(a^2+b^2)e^*_1\otimes e^*_1+(b-ayx^{-1})e^*_1\odot e^*_2+ae^*_1\odot e^*_3+\e\sqrt{x^2+y^2+1}e^*_1\odot e^*_5\\
				&&(1+(yx^{-1})^2)e^*_2\otimes e^*_2-yx^{-1}e^*_2\odot e^*_3+e^*_3\otimes e^*_3+x^2 e^*_4\otimes e^*_4.\;(x\not=0,a,b,y\in \R,\e=\pm 1).
			\end{eqnarray*}	
			
			\item[$(d)$] $(L_{5,3},\prs_{5,3})$ with  
			\begin{eqnarray*}
				\prs_{5,3}
				&=&e^*_1\otimes e^*_1+ae^*_1\odot e^*_2+(a^2+b^2)e^*_2\otimes e^*_2+b e^*_2\odot e^*_3+\e\sqrt{x^2+1}  e^*_2\odot e^*_4\\&&+(1+x^2)e^*_3\otimes e^*_3- x e^*_3\odot e^*_5+e^*_5\otimes e^*_5,\;(x,a,b\in\R,\e=\pm 1).
			\end{eqnarray*}	
			
			\item[$(e)$] $(L_{5,5},\prs_{5,5,1})$ or $(L_{5,5},\prs_{5,5,2})$ with 
			\begin{eqnarray*}
				\prs_{5,5,1}&=&(a^2+b^2)e^*_1\otimes e^*_1+a\rho^{-1}e^*_1\odot e^*_2+\rho(b-ax^{-1}y) e^*_1\odot e^*_4+\sqrt{x^2+y^2}e^*_1\odot e^*_5\\
				&&+\rho^{-2}e^*_2\otimes e^*_2-x^{-1}ye^*_2\odot e^*_4+x^2\rho^{-2}e^*_3\otimes e^*_3+\rho^2(1+(x^{-1}y)^2)e^*_4\otimes e^*_4,\\&&\quad (x\not=0,\rho\not=0,a,b,y\in\R)
			\end{eqnarray*}	or
			\begin{eqnarray*}
				\prs_{5,5,2}&=&e^*_1\otimes e^*_1+be^*_1\odot e^*_2+(a^2+b^2)e^*_2\otimes e^*_2+ae^*_2\odot e^*_3+\e\sqrt{x^2+1} e^*_2\odot e^*_5\\
				&&(1+x^2)e^*_3\otimes e^*_3+x\rho e^*_3\odot e^*_4+\rho^{2}e^*_4\otimes e^*_4,\quad
				(\rho\not=0,x,a,b\in\R,\e=\pm 1).
			\end{eqnarray*}

			\item[$(f)$] $(L_{5,6},\prs_{5,6})$ with 
			\begin{eqnarray*}
				\prs_{5,6}&=&(a^2+b^2)e^*_1\otimes e^*_1+(b+ax^{-1}y)e^*_1\odot e^*_2+\mu ae^*_1\odot e^*_3+\e \mu^2\sqrt{x^2+y^2+1} e^*_1\odot e^*_5\\
				&&+(1+x^{-2}y^2)e^*_2\otimes e^*_2+\mu x^{-1}ye^*_2\odot e^*_3+\mu^2 e^*_3\otimes e^*_3+\mu^4x^2 e^*_4\otimes e^*_4,\\
				&&\mu\not=0,x\not=0,a,b,y\in\R,\e=\pm 1. 
			\end{eqnarray*}

		\end{enumerate}
		
\end{theo}

\begin{proof} According to Theorems \ref{main} and \ref{five}, $\G=\R e\oplus V\oplus \R\bar{e}$, where $(V,\prs_0)$ is a 3-dimensional Euclidean vector space.  The Lie bracket is given by
	\[ [\bar{e},u]=Du+\langle b,u\rangle_0 e\esp [u,v]=\langle Ku,v\rangle_0 e, \quad u,v\in V,\] $e$ is central,   $b\in V,$ $K,D:V\too V$ with $K$ skew-symmetric, $D$ is nilpotent, $KD+D^*K=0$ and $\tr(K^2)=-2\tr(D^*D)$ moreover the metric $\prs$ satisfies $\prs_{|V}=\prs_0$, $e$ and $\bar{e}$ are isotropic in duality and are orthogonal to $V$.

	$\bullet$	If $K=D=0$ then there exists a  Lorentzian basis $(\bar{e},e,f_1,f_2,f_3)$ such that 
	\[ [\bar{e},f_1]=\al e,\; [\bar{e},f_2]=\be e\esp [\bar{e},f_3]=\ga e,\;\al\not=0. \]	Put
	$$(e_1,e_2,e_3,e_4,e_5)=(\e\bar{e},f_1,|\al| e,\mu_1^{-1}(f_2-\frac{\be}{\al}f_1),\mu_2^{-1}(f_3-\frac{\ga}{\al}f_1)),$$ 
	where $\e$ is the sign of $\al$,  $\mu_1=||f_2-\frac{\be}{\al}f_1||$ and $\mu_2=||f_3-\frac{\ga}{\al}f_1||$.	Thus $(\G,\br,\prs)$ is isomorphic to $(L_{5,2},\prs_{5,2})$ with  
	$$\prs_{5,2}=\al e^*_1\odot e^*_3+e^*_2\otimes e^*_2+e^*_4\otimes e^*_4+e^*_5\otimes e^*_5+a e^*_2\odot e^*_4+b e^*_2\odot e^*_5+ab e^*_4\odot e^*_5,$$
	where $\al\not=0$, $a=\frac{\be_1}{\sqrt{1+\be_1^2}}$, $b=\frac{\ga_1}{\sqrt{1+\ga_1^2}}$, $\be_1=\frac{\be}{\al}$ and $\ga_1=\frac{\ga}{\al}$.
	So $|a|<1$ and $|b|<1$.
	
	$\bullet$	If $K\not=0$ then according to Lemma \ref{KD},   there exists an orthonormal basis $\B_0=(f_1,f_2,f_3)$ of $V$ such that the matrices of $K$, $S$ and $D$ are given by
	\[ M(S,\B_0)=\mathrm{Diag}(0,a),\; M(K,\B_0)=\left(\begin{matrix}0&0&0\\
	0&0&-\al\\0&\al&0
	\end{matrix}  \right)\esp M(D,\B_0)=\left(\begin{matrix}0&x&y\\
	0&0&a\al^{-1}\\0&0&0
	\end{matrix}  \right),\;\al>0. \]Put $c=a\al^{-1}$. The condition $\tr(K^2)=-2\tr(D^*D)$ gives $\al=\sqrt{x^2+y^2+c^2}$. 
	Thus the Lie bracket is given by
	\[ [\bar{e},f_1]=\ga e,\;[\bar{e},f_2]=x f_1+\mu e,\;[\bar{e},f_3]=yf_1+cf_2+\be e\esp [f_2,f_3]=\al e. \]
	Put $a=-\be\al^{-1}$, $b=\mu\al^{-1}$, $z=\al e$ and $\bar{z}=\bar{e}+af_2+bf_3$.
	We get	\[ [\bar{z},f_1]=\ga\al^{-1} z,\;[\bar{z},f_2]=x f_1,\;[\bar{z},f_3]=yf_1+cf_2\esp [f_2,f_3]=z. \]
	
	$\bullet$ \underline{$\ga=0$,  $x\not=0$ and $c=0$}. Then
	\[ \;[\bar{z},f_2]=x f_1\esp [f_2,f_3-yx^{-1}f_2]=z.\]
	Put $(e_1,e_2,e_3,e_4,e_5)=(f_2,\bar{e}+af_2+bf_3,f_3-yx^{-1}f_2,-xf_1,\al e)$
	Thus $(\G,\br,\prs)$ is isomorphic to  $(L_{5,8},\prs_{5,8})$.

	$\bullet$ \underline{$\ga=0$,  $x\not=0$ and $c\not=0$}.
	\[ \;[\bar{z},f_2]=x f_1,\;[\bar{z},f_3-yx^{-1}f_2]=cf_2\esp [f_2,f_3-yx^{-1}f_2]=z.\]
	Put $$(e_1,e_2,e_3,e_4,e_5)=(c^{-1}(\bar{e}+af_2+bf_3),f_3-yx^{-1}f_2,f_2,c^{-1}xf_1,-\al e)).$$ After the change of parameters $c^{-1}(a,b,x,y)$ to $(a,b,x,y)$, we get that
	$(\G,\br,\prs)$ is isomorphic to $(L_{5,9},\prs_{5,9})$.

	\underline{	$\bullet$ $\ga=0$,  $x=0$, $c=0$}. Put
	\[ (e_1,e_2,e_3,e_4,e_5)=(f_3,\bar{e}+af_2+bf_3,-f_2,-yf_1,\al e). \]
	Thus $(\G,\br,\prs)$ is isomorphic to $(L_{5,8},\prs_{5,8})$ with $y=0$.

	\underline{	$\bullet$ $\ga=0$,  $x=0$, $c\not=0$}. Put 
	\[ (e_1,e_2,e_3,e_4,e_5)=(f_3,c^{-1}(\bar{e}+af_2+bf_3),-f_2-c^{-1} yf_1,\al e,f_1). \]After the change of parameters $c^{-1}(a,b,y)$ to $(a,b,y)$ we get that
	$(\G,\br,\prs)$ is isomorphic to 
	$(L_{5,3},\prs_{5,3})$.

	$\bullet$ \underline{ $\ga\not=0$}. Put    $g_1=\al\ga^{-1}f_1$ then
	\[ [\bar{z},g_1]= z,\;[\bar{z},f_2]=x\al^{-1}\ga g_1,\;[\bar{z},f_3]=y\al^{-1}\ga g_1+cf_2\esp [f_2,f_3]=z. \]

	\underline{$\ga\not=0$ and $c=0$}. Then $(x,y)\not=(0,0)$ and we can suppose that $x\not=0$. Then
	\[ [\bar{z},g_1]= z,\;[\bar{z},f_2]=x\al^{-1}\ga g_1,\;[\bar{z},f_3- x^{-1}y f_2]=0\esp [f_2,f_3- x^{-1}y f_2]=z. \]
	Put
	\[ (e_1,e_2,e_3,e_4,e_5)=(\bar{e}+af_2+bf_3,x^{-1}\al\ga^{-1}f_2,\al\ga^{-1}f_1,x\al^{-1}\ga(f_3- x^{-1}y f_2),\al e)\esp \rho=x\al^{-1}\ga. \]
	Then $(\G,\br,\prs)$ is isomorphic to$(L_{5,5},\prs_{5,5,1})$. 
	
	\underline{$\ga\not=0$, $c\not=0$ and $x=0$.} Then
	\[ [\bar{z},g_1]= z,\;[\bar{z},f_3]=y\al^{-1}\ga g_1+cf_2\esp [f_2,f_3]=z. \]
	Put 
	\[ (e_1,e_2,e_3,e_4,e_5)=(-f_3,c^{-1}(\bar{e}+af_2+bf_3),f_2+c^{-1}y\al^{-1}\ga g_1,cg_1,\al e). \]
	After the change of parameters $c^{-1}(a,b,y)$ to $(a,b,x)$ and $\rho=c\al\ga^{-1}$ we get that $(\G,\br,\prs)$ is isomorphic to $(L_{5,5},\prs_{5,5,2})$.

	\underline{$\ga\not=0$, $c\not=0$ and $x\not=0$.} Then
	\[	[c^{-1}\bar{z},cg_1]= z,\;[c^{-1}\bar{z},f_2]=c^{-1}x\al^{-1}\ga g_1,\;[c^{-1}\bar{z},f_3-x^{-1}yf_2]=f_2\esp [f_2,f_3-x^{-1}yf_2]=z. \]
	Put
	\[ (e_1,e_2,e_3,e_4,e_5)=(-c^{-1}(\bar{e}+af_2+bf_3),f_3-x^{-1}yf_2,-f_2,-cg_1,\al e). \]
	Then
	\[ [e_1,e_2]=e_3,[e_1,e_3]=ke_4,[e_1,e_4]=e_5,[e_2,e_3]=e_5. \]			
	We can always suppose that $k>0$ (otherwise replace $e_3$ by $-e_3$ and $e_2$ by $-e_2$). We put $e_1'=\mu e_1$ and $e_3'=\mu e_3$, $e_5'=\mu e_5$ and  $\mu^2=\frac{1}{{k}}$. After an adequate change of parameters one can see that $(\G,\br,\prs)$ is isomorphic to $(L_{5,6},\prs_{5,6})$. 
\end{proof}

$\ $
\begin{exem}\label{exem}$\ $ \begin{enumerate}

		\item { Example of a six dimensional Ricci flat Lorentzian nilpotent Lie algebra with nondegenerate center.}
	\[ [e_1,e_3]=e_6,\;[e_1,e_5]=e_6,\;[e_2,e_3]=-e_6,[e_2,e_4]=e_6,[e_3,e_4]=e_1,[e_3,e_5]=e_2\esp [e_4,e_5]=e_1+e_2. \]
	$\B=(e_1,\ldots,e_6)$ is an orthonormal basis with $\langle e_1,e_1\rangle=-1$.	
	\item { Example of a seven dimensional Ricci flat Lorentzian nilpotent Lie algebra with nondegenerate center.}
	\[ [e_1,e_3]=\sqrt{2}e_7,\;[e_2,e_4]=\sqrt{2}e_7
	,\;[e_4,e_5]=-e_1,[e_4,e_6]=-e_1,[e_3,e_5]=-e_2,[e_3,e_6]=-e_2. \]
	$\B=(e_1,\ldots,e_7)$ is an orthonormal basis with $\langle e_1,e_1\rangle=-1$.
	
	\item { Example of an eight dimensional Einstein Lorentzian nilpotent Lie algebra with non vanishing scalar curvature.} This example was given in \cite{Dconti1}.
	\[ \begin{cases}
	[e_1,e_2]=-4\sqrt{3}e_3,\;[e_1,e_3]=\sqrt{\frac52}e_4,\;
	[e_1,e_4]=-2\sqrt{3}e_8,\;[e_1,e_5]=3\sqrt{\frac72}e_6,;\\
	[e_1,e_6]=-4\sqrt{2}e_7,\;
	[e_2,e_3]=-\sqrt{\frac52}e_5,\;[e_2,e_4]=-3\sqrt{\frac72}e_6,\;
	[e_2,e_5]=-2\sqrt{3}e_7,\;
	\\\;[e_2,e_6]=-4\sqrt{2}e_8,\;
	[e_3,e_4]=-\sqrt{21}e_7,\;[e_3,e_5]=-\sqrt{21}e_8.
	\end{cases} \]
	$\B=(e_1,\ldots,e_8)$ is an orthonormal basis with $\langle e_6,e_6\rangle=-1$.

	\end{enumerate}
	
\end{exem}
\begin{table}[!htbp]
	\small
	\begin{center}
		\begin{tabular}{|l|l|l|}
			\hline
			Lie Algebra        &  Lie brackets                                                       & Non Trace-free Derivation                                                                                                \\ \hline
			$\mathrm{L}_{3,2}$ & $[e_1,e_2]=e_3$                                                    & $e^1\otimes e_1+e^3\otimes e_3$                                                                                          \\\hline
			$\mathrm{L}_{4,2}$ & $[e_1,e_2]=e_3$                                                    & $e^1\otimes e_1+e^3\otimes e_3$                                                                                          \\\hline
			$\mathrm{L}_{4,3}$ & $[e_1,e_2]=e_3$, $[e_1,e_3]=e_4$                                   & $2e^2\otimes e_2-e^1\otimes e_1+e^3\otimes e_3$                                                                          \\\hline
			$\mathrm{L}_{5,2}$ & $[e_1,e_2]=e_3$                                                    & $e^1\otimes e_1+e^3\otimes e_3$                                                                                          \\\hline
			$\mathrm{L}_{5,3}$ & $[e_1,e_2]=e_3$, $[e_1,e_3]=e_4$                                   & $2e^2\otimes e_2-e^1\otimes e_1+e^3\otimes e_3$                                                                          \\\hline
			$\mathrm{L}_{5,4}$ & $[e_1,e_2]=e_5$, $[e_3,e_4]=e_5$                                   & $e^1\otimes e_1+e^3\otimes e_3+e^5\otimes e_5$                                                                           \\\hline
			$\mathrm{L}_{5,5}$ & $[e_1,e_2]=e_3$, $[e_1,e_3]=e_5$, $[e_2,e_4]=e_5$                  & $e^3\otimes e_3+2e^2\otimes e_2+2e^5\otimes e_5-e^1\otimes e_1$                                                          \\\hline
			$\mathrm{L}_{5,6}$ & $[e_1,e_2]=e_3$, $[e_1,e_3]=e_4$, $[e_1,e_4]=e_5$, $[e_2,e_3]=e_5$ & $e^1\otimes e_1+2e^2\otimes e_2+3e^3\otimes e_3+4e^4\otimes e_4+5e^5\otimes e_5$ \\\hline
			$\mathrm{L}_{5,7}$ & $[e_1,e_2]=e_3$, $[e_1,e_3]=e_4$, $[e_1,e_4]=e_5$                  & $e^1\otimes e_1-2e^2\otimes e_2-e^3\otimes e_3+e^5\otimes e_5$                                                        \\\hline
			$\mathrm{L}_{5,8}$ & $[e_1,e_2]=e_4$, $[e_1,e_3]=e_5$                                   & $e^1\otimes e_1-e^2\otimes e_2+e^5\otimes e_5$                                                                           \\\hline
			$\mathrm{L}_{5,9}$ & $[e_1,e_2]=e_3$, $[e_1,e_3]=e_4$, $[e_2,e_3]=e_5$                  & $2e^1\otimes e_1-e^2\otimes e_2+e^3\otimes e_3+3e^4\otimes e_4$                                                          \\ \hline
		\end{tabular}
	\end{center}
	\caption{Table of nilpotent Lie algebras of dimension $\leq 5$ with non null trace  derivation }
	\label{tab:my-table}
\end{table}


\begin{thebibliography}{99}
	
	\bibitem{bou1} M. Ait Haddou, M. Boucetta, H. Lebzioui,  Left-invariant Lorentzian flat metrics on Lie groups,
	Journal of Lie Theory {\bf 22} (2012), No. 1, 269-289.
	
	
	
	
	\bibitem{Aub-Med}  \textsc{ Aubert Anne \& Medina Alberto,} Groupes de Lie
	pseudo-riemanniens plats.  Tohoku Math. J. {\bf (2) 55} (2003), no. 4, 487-506.
	
	\bibitem{besse} Arthur L. Besse, Einstein manifolds, Classic in Mathematics Springer (2008).

	\bibitem{bouc1}
	 Boucetta, M., Ricci flat left invariant Lorentzian metrics on
		2-step nilpotent Lie groups, arXiv:0910.2563v2[math.DG] 15 Feb 2010.
	
\bibitem{bouh} Mohamed Boucetta and Hicham Lebzioui, On flat pseudo-Euclidean nilpotent Lie algebras,	Journal of Algebra {\bf537} (2019) 459-477
	
	\bibitem{calvaruso3} Giovanni Calvaruso and Amirhesam Zaeim, Four-dimensional homogeneous Lorentzian manifolds, Monatsh Math  {\bf 174} (2014) 377-402.
	
	\bibitem{Dconti1} Conti, D., Rossi, F. A. (2019). Einstein nilpotent Lie groups. Journal of Pure and Applied Algebra, {\bf 223(3)}, 976-997.
	
	\bibitem{graaf} W.A. De Graaf, Classification of 6-dimensional nilpotent Lie algebras over fields of characteristic not 2, J. Algebra {\bf 309}
	(2007) 640-653.
	
	
	\bibitem{derd} A. Derdzinski, S.R. Gal, Indefinite Einstein metrics on simple Lie groups, Indiana Univ. Math. J. {\bf63 (1)} (2014) 165-212.

	\bibitem{guediri} Mohammed Guediri \& Mona Bin-Asfour,
Ricci-flat left-invariant Lorentzian metrics on 2-step nilpotent Lie groups,
Archivum Mathematicum, {\bf Vol. 50} (2014), No. 3, 171-192.

	
	\bibitem{heber} Heber J., Noncompact Einstein spaces, Invent. Math. {\bf 133} (1998) 279-352.
	
		
	\bibitem{lauret} J. Lauret, A canonical compatible metric for geometric structures on nilmanifolds, Ann. Global Anal. Geom. {\bf 30} (2006) 107-138.
	
	\bibitem{lauret1} J. Lauret, Einstein solvmanifolds are standard, Annals of Mathematics, {\bf 172} (2010), 1859-1877.
	
	
	
	\bibitem{medina} Alberto Medina \& Philippe Revoy, Alg\`ebre de Lie et produit scalaire invariant,
	Annales scientifiques de l'\'Ecole Normale Sup\'erieure (1985)
	{\bf Vol. 18}, Issue: 3, page 553-561.
	
	
	\bibitem{milnor} \textsc{ Milnor J.},  Curvature of left invariant metrics on Lie
	groups, Adv. in Math. {\bf 21} (1976), 283-329.
	
	 
	
	
	
	
	
\end{thebibliography}
\end{document}